\documentclass[11pt]{article}
\usepackage[intlimits]{amsmath}
\usepackage{latexsym,amsfonts,amsthm,amssymb, epsfig}
\usepackage{color,comment}
\usepackage[vcentering,dvips]{geometry}
\newtheorem{thm}{Theorem}
\newtheorem*{thm*}{Theorem}
\newtheorem{lem}[thm]{Lemma}

\theoremstyle{remark}

\theoremstyle{definition}
\newtheorem{dfn}[thm]{Definition}

\newcommand{\R}{\mathbb{R}}
\newcommand{\N}{\mathbb{N}}
\newcommand{\C}{\mathbb{C}}

\newcommand{\vol}{\operatornamewithlimits{vol}}
\newcommand{\re}{\operatornamewithlimits{Re}}
\newcommand{\im}{\operatornamewithlimits{Im}}

\title{Entropy numbers of finite-dimensional embeddings}
\author{Marta Kossaczk\'a\footnote{Department of Mathematical Analysis, Charles University, Sokolovsk\'a 83, 186 00, Prague 8, Czech Republic},
Jan Vyb\'iral\footnote{Dept. of Mathematics FNSPE, Czech Technical University in Prague, Trojanova 13, 12000 Prague, Czech Republic;
\texttt{jan.vybiral@fjfi.cvut.cz}}}

\begin{document}
\maketitle
\begin{abstract}
Entropy numbers and covering numbers of sets and operators are well known geometric notions, which
found many applications in various fields of mathematics, statistics, and computer science.
Their values for finite-di\-men\-sion\-al embeddings $id:\ell_p^n\to \ell_q^n$,
$0<p,q\le\infty$, are known (up to multiplicative constants) since the pioneering work of Sch\"utt in 1984, with later improvements by
Edmunds and Triebel, K\"uhn and Gu\'edon and Litvak. The aim of this survey is to give a self-contained presentation
of the result and an overview of the different techniques used in its proof.
\end{abstract}

\section{Introduction}

The concept of covering numbers can be traced back to the work of Kolmogorov \cite{Ko1, Ko2}
and became since then an important tool used in several areas of theoretical and applied
mathematics. Furthermore,  Pietsch introduced in his book \cite{Pietsch} a formal definition
of an inverse function of covering numbers under the name of entropy numbers.
Later on, Carl and Triebel \cite{BC,CaTr} investigated the relation between entropy numbers
and other geometric and approximation quantities related to sets and operators, most importantly to eigenvalues of compact operators.

Due to the natural definition of covering and entropy numbers, and due to their relations to other geometric notions,
these concepts found applications in many areas of pure and applied mathematics,
including geometry of Banach spaces  \cite{BLM,apl:geom:1,apl:geom:2,apl:geom:3,Pisier},
information theory \cite{apl:ee:4,RWY,apl:ee:2,apl:ee:1,apl:ee:3}, and random processes \cite{LT,apl:stat:1}.
They also appeared in the theory of compressed sensing \cite{CRT,D} in the study of optimality
of recovery of sparse vectors and in the study of eigenvalue problems in Banach spaces \cite{ET,Kon}.
One of the most important classes of operators, whose entropy numbers are well understood and often applied,
are the identities between finite-dimensional vector spaces.
The main aim of this note is to present a self-contained overview of this area.
To state the main result in detail, we need to recall some notation.

The couple $(X,\|\cdot\|_X)$
is called a quasi-Banach space, if $X$ is a real or complex vector space and the mapping $\|\cdot\|_X:X\to [0,\infty)$ satisfies
\begin{enumerate}
\item[(i)] $\|x\|_X=0$ if, and only if, $x=0$,
\item[(ii)] $\|\alpha x\|_X=|\alpha|\cdot \|x\|_X$ for all $\alpha\in\R$ (or in $\C$) and all $x\in X$,
\item[(iii)] there exists a constant $C\ge 1$ such that $\|x+y\|_X\le C(\|x\|_X+\|y\|_X)$ for all $x,y\in X$,
\item[(iv)] $X$ is complete with respect to $\|\cdot\|_X.$ 
\end{enumerate}
If the constant $C$ in (iii) can be chosen to be equal to one, $X$ is actually a Banach space. If $\|\cdot\|_X$
satisfies the axioms above with (iii) replaced by 
$$
\|x+y\|^p_X\le \|x\|_X^p+\|y\|_X^p,\quad x,y\in X
$$
for some $0<p\le 1$, then $(X,\|\cdot\|_X)$ is called a $p$-Banach space and $\|\cdot\|_X$ is a $p$-norm.
It follows that a Banach space $X$ is also a $p$-Banach space for $p=1$.
It is easy to see that every $p$-norm is a quasi-norm with $C=2^{1/p-1}$. On the other hand, by the Aoki-Rolewicz theorem \cite{Aoki,Rol}, every quasi-norm
is equivalent to some $p$-norm for a suitably chosen $p$. We refer to \cite{Kalt} for a survey on quasi-Banach spaces.

If $X$ is a vector space equipped with some (quasi-)norm or $p$-norm $\|\cdot\|_X$, we denote by $B_X$
its unit ball, i.e. the set $B_X=\{x\in X:\|x\|_X\le 1\}.$ The symbol ${\mathcal L}(X,Y)$ stands for the set of all bounded linear operators from $X$ to $Y$.
For $0<p\le\infty$, we define $\ell_p^n(\R)$ (or $\ell_p^n(\C)$) to be the Euclidean space $\R^n$ (or $\C^n$) equipped with the (quasi-)norm
\[
\|x\|_p=\|(x_i)_{i=1}^n\|_p=\begin{cases}
\displaystyle \Bigl(\sum_{i=1}^n|x_i|^p\Bigr)^{1/p}\,,& 0< p < \infty\,; \\
\displaystyle\max_{1\leq i \leq n}|x_i|\,,& p=\infty.
\end{cases}
\]
The unit ball of $\ell_p^n(\R)$ will be denoted by $B_p^n$.
It is easy to see, that $\ell_p^n(\R)$ and $\ell_p^n(\C)$ are Banach spaces if $p\ge 1$ and $p$-Banach spaces if $0<p\le 1.$
Therefore, we will denote $\bar p=\min(1,p)$ and use the triangle inequality in $\ell_p^n(\R)$ and $\ell_p^n(\C)$ in the form
$$
\|x+y\|_p^{\bar p}\le \|x\|_p^{\bar p}+\|y\|_p^{\bar p}.
$$

We define now the concept of entropy numbers of a bounded linear operator $T$ between two (quasi-)Banach spaces $X$ and $Y$. Essentially, we are allowed to use
$2^{k-1}$ balls of radius $r$ in $Y$ to cover the image of the unit ball of $X$ by $T$ and $e_k(T)$ denotes the smallest $r$, for which this is still possible.

\begin{dfn}\label{defn:entropy}
Let $X$ and $Y$ be Banach spaces, $p$-Banach spaces, or quasi-Banach spaces. Let $T:X\to Y$ be a bounded linear operator and let $k\ge 1$ be a 
positive integer.
The $k^{\rm th}$ (dyadic) entropy number of $T$ is defined as
\begin{equation}\label{eq:entropy:def1}
e_k(T):=\inf\Bigr\{r>0:\exists y_1,\dots,y_{2^{k-1}}\in Y\ \text{with}\ T(B_X)\subset \bigcup_{j=1}^{2^{k-1}}(y_j+r B_Y)\Bigl\}.
\end{equation}
\end{dfn}

The relation of the entropy numbers to the covering numbers of Kolmogorov is quite straightforward.
If $K\subset Y$ and $r>0$, then the covering number $N(K,Y,r)$ is the smallest number $N$ such that there exist points
$y_1,\dots,y_N$ with $K\subset\bigcup_{j=1}^N (y_j+rB_Y).$ The entropy numbers $e_k(T)$ can then be equivalently defined as
$$
e_n(T)=\inf\{r>0:N(T(B_X),Y,r)\le 2^{k-1}\}.
$$

Although easy to define, the entropy numbers of some specific operator $T$ are usually rather difficult to calculate, or estimate.
One class of operators, where the upper and lower bounds on entropy numbers are known, are the identities between finite-dimensional vector spaces.
The main aim of this note is to present a self-contained proof of the following result.
\begin{thm}\label{hlavna1''}
Let $0<p,q\leq \infty$ and let $n\in\N$.
\begin{itemize}
\item[a)] If $0<p\leq q\leq \infty $ then for all $k\in \mathbb{N}$ it holds
\begin{equation}\label{eq:2ekviv1''}
e_k (id:\ell^n_p(\mathbb{R})\rightarrow \ell^n_q(\mathbb{R}))\sim 
\begin{cases} 1 & \text{if} \quad 1\leq k\leq \log_2n,\\
\displaystyle\biggl(\frac{\log_2(1+n/k)}{k}\biggr)^{\frac{1}{p}-\frac{1}{q}} & \text{if} \quad \log_2n\leq k\leq n,\\[10pt]
\displaystyle2^{-\frac{k-1}{n}}n^{\frac{1}{q}-\frac{1}{p}} & \text{if} \quad n\leq k.
\end{cases}
\end{equation}
\item[b)] If $0<q\leq p\leq \infty$ then for all $k\in \mathbb{N}$ it holds
\begin{equation}\label{eq:2ekviv2''}
 e_k (id:\ell^n_p(\mathbb{R})\rightarrow \ell^n_q(\mathbb{R}))\sim  2^{-\frac{k-1}{n}}n^{\frac{1}{q}-\frac{1}{p}}.
\end{equation}
\end{itemize}
The constants of equivalence in both \eqref{eq:2ekviv1''} and \eqref{eq:2ekviv2''} may depend on $p$ and $q$, but are independent of $k$ and $n$.
\end{thm}

Theorem \ref{hlavna1''} was first proved for $1\le p<q\le\infty$ by Sch\"utt \cite{S} with partial results given before by
H\"ollig \cite{H} and Pietsch \cite[Section 12.2]{Pietsch}. The case of quasi-Banach spaces (i.e. when $p$ and/or $q$ is smaller than one)
was studied by Edmunds and Triebel in \cite[Section 3.2.2]{ET}, who provided
the estimates from above in \eqref{eq:2ekviv1''} including the intermediate $k$'s with $\log_2n\le k\le n.$
Finally, the corresponding lower bound was supplied by K\"uhn \cite{K2001} and, independently, by Gu\'edon and Litvak in \cite{GL}.
Although all the estimates in Theorem \ref{hlavna1''} may be found in the literature since nearly two decades, there seems to be no
place, were all the parts would appear together - the reader has to combine several sources, sometimes using different notations.

Interestingly enough, the proof of Theorem \ref{hlavna1''} requires a combination of several different techniques, which are of independent interest.
The most natural are the so-called volume arguments. Quite intuitively, if we want to cover a body $K\subset \R^n$ by a number of other bodies $L_1,\dots,L_N\subset\R^n$,
then their volume combined must be larger than the volume of $K$. Here, volume can be any positive measure on $\R^n$.
As $L_j$'s are now different translates of a fixed dilation of one fixed body $L$,
it is most convenient to work with a shift-invariant measure, which behaves well with respect to dilations. It is therefore most natural
to work with the usual Lebesgue measure on $\R^n$ and $\vol(K)$ will be the usual Lebesgue volume of a measurable set $K\subset \R^n.$
If $K=B_X$ and $L=B_Y$ are unit balls of some quasi-Banach spaces $X$ and $Y$, and if the number $N$ is fixed to be $N=2^{k-1}$ for a positive integer $k$,
this can be immediately translated into lower bounds of the entropy numbers $e_k(id:X\to Y)$.

With a bit of additional work, volume arguments can be also used to give upper bounds on entropy numbers. Indeed, if $r>0$ is fixed,
we take a maximal set $y_1,\dots,y_N\in B_X\subset Y$ with $\|y_i-y_j\|_Y\ge r.$ Then, on one hand, the sets $y_j+c_1rB_Y$ are disjoint
and $B_X$ is covered by the union of $y_j+c_2rB_Y$ for suitably chosen constants $c_1,c_2>0.$ On the other hand,
the sets $y_j+c_1rB_Y$ are included in a certain multiple of $B_X$, which must therefore have a larger volume than all the disjoint sets $y_j+c_1rB_Y$
combined. This gives an upper bound on their number $N$, leading to an upper bound on the entropy numbers.

Although the volume arguments represent a powerful technique, which can in principle be applied to all the parameter settings in Theorem \ref{hlavna1''},
the obtained bounds are not always optimal. Actually, it turns out that the volume arguments provide optimal bounds (up to a multiplicative constant)
only when $k$ is large or, equivalently, when $r$ is small. For smaller $k$'s,
direct combinatorial estimates are needed to provide both the lower and the upper bounds.
Geometrically it means, that any good cover of $T(B_X)$ with a small number of sets $y_j+rB_Y$
needs to have big overlap and/or to cover also some large neighborhood of $T(B_X).$

The structure of the paper is as follows. Section \ref{sec:2} gives basic properties of entropy numbers and Gamma function,
presents the calculation of the volume of $B_p^n$, and provides a couple of lemmas used later. The proof of Theorem \ref{hlavna1''}
comes in Section \ref{sec:3}. Finally, Section \ref{sec:4} collects few additional remarks and topics, including the extension
of Theorem \ref{hlavna1''} to the complex setting.

\section{Preparations}\label{sec:2}

We start by recalling few well-known basic facts about entropy numbers.
Although the reader may find the proof, for instance, in \cite{CaSt} or \cite{ET}, we include it for the sake of completeness.
\begin{thm}\label{thm:entropy}Let $X,Y,Z$ be $p$-Banach spaces for some $0<p\le 1.$
Let $S,T\in{\mathcal L}(X,Y)$ and $R\in{\mathcal L}(Y,Z)$. Then, for all $k,l\in\N$, it holds
\begin{enumerate}
\item[(i)] $\|T\|\geq e_1(T) \geq e_2(T) \geq \dots \geq 0$ (monotonicity);
\item[(ii)] $e^p_{k+l-1}(S+T) \leq e^p_k(S)+e^p_l(T)$ (subadditivity);
\item[(iii)] $e_{k+l-1}(R\circ T) \leq e_k(R)\cdot e_l(T)$ (submultiplicativity).
\end{enumerate}
\end{thm}
\begin{proof}Monotonicity of entropy numbers follows directly from \eqref{eq:entropy:def1}. Similarly, $e_1(T)\le \|T\|$ is
implied by $T(B_X)\subset \|T\|\cdot B_Y.$

For the proof of subadditivity, let $k,l$ be positive integers and let $\varepsilon>0.$ Then there are $y_1,\dots,y_{2^{k-1}}\in Y$
and $z_1,\dots,z_{2^{l-1}}\in Y$ with
$$
S(B_X)\subset \bigcup_{i=1}^{2^{k-1}}\Bigl(y_i+(e_k(S)+\varepsilon) B_Y\Bigr)\quad\text{and}\quad
T(B_X)\subset \bigcup_{j=1}^{2^{l-1}}\Bigl(z_j+(e_l(T)+\varepsilon) B_Y\Bigr).
$$
Then
\begin{align*}
(S+T)(B_X)&\subset S(B_X)+T(B_X)\\
&\subset\biggl[\bigcup_{i=1}^{2^{k-1}}\Bigl(y_i+(e_k(S)+\varepsilon) B_Y\Bigr)\biggr]+ \biggl[\bigcup_{j=1}^{2^{l-1}}\Bigl(z_j+(e_l(T)+\varepsilon) B_Y\Bigr)\biggr]\\
&=\bigcup_{i,j}\Bigl(y_i+z_j+(e_k(S)+\varepsilon) B_Y+(e_l(T)+\varepsilon) B_Y\Bigr)\\
&\subset\bigcup_{i,j}\Bigl(y_i+z_j+[(e_k(S)+\varepsilon)^p+(e_l(T)+\varepsilon)^p]^{1/p} B_Y\Bigr)
\end{align*}
and taking the infimum over $\varepsilon>0$ gives the result.

The proof of submultiplicativity follows in a similar way. Indeed, if $k,l$ are positive integers and $\varepsilon>0$,
we find $y_1,\dots,y_{2^{l-1}}\in Y$ and $z_1,\dots,z_{2^{k-1}}\in Z$ with
$$
T(B_X)\subset \bigcup_{i=1}^{2^{l-1}}\Bigl(y_i+(e_l(T)+\varepsilon) B_Y\Bigr)\quad\text{and}\quad R(B_Y)\subset \bigcup_{j=1}^{2^{k-1}}\Bigl(z_j+(e_k(R)+\varepsilon) B_Z\Bigr).
$$
The proof then follows from
\begin{align*}
(R\circ T)(B_X)&=R(T(B_X))\subset R\biggl(\,\bigcup_{i=1}^{2^{l-1}}\Bigl(y_i+(e_l(T)+\varepsilon) B_Y\Bigr)\biggr)\\
&=\bigcup_{i=1}^{2^{l-1}}\Bigl(Ry_i+(e_l(T)+\varepsilon) R(B_Y)\Bigr)\\
&\subset \bigcup_{i,j}\Bigl(Ry_i+(e_l(T)+\varepsilon)z_j+(e_l(T)+\varepsilon)(e_k(R)+\varepsilon) B_Z\Bigr).\qedhere
\end{align*}
\end{proof}

We will also need few basic facts about the Gamma function, which is defined by
$\Gamma(t)=\int_0^\infty x^{t-1}e^{-x}dx$ for every positive real number $t>0$. By partial integration, we get
\begin{equation}\label{eq:Gamma:ind}
\Gamma(1+t)=t\,\Gamma(t) \quad \text{for every}\quad t>0
\end{equation}
and $\Gamma(n)=(n-1)!$ for every positive integer $n.$ Furthermore, by standard calculus it follows that $\Gamma$ is a continuous function on $(0,\infty)$.

\begin{lem}\label{lem:Gamma}Let $0<p<\infty$. Then
\begin{equation}\label{eq:Gamma_1}
\Gamma(1+x/p)^{1/x}\sim x^{1/p},\quad x\ge 1,
\end{equation}
where the constants of equivalence may depend on $p$.
\end{lem}
\begin{proof} The result is a corollary of Stirling's formula \cite[Chapter A.2, Theorem 2.3]{Stein}, but we give also a simple proof for reader's convenience.

If $x=n$ is a positive integer and $p=1$, the result follows from
\begin{equation}\label{eq:Gamma_2}
n\log(n)-n\le \log(n!)=\log(\Gamma(1+n))\le n\log(n),
\end{equation}
where the right-hand side of \eqref{eq:Gamma_2} comes from $n!\le n^n$ and the left-hand side can be obtained by
taking the Riemann sum of $\int_1^n\log(x)dx$ at right endpoints.

We modify this strategy also for $x\ge 1$ and $p>0.$ Let $k\in\N_0$ be the unique integer with
$$
\frac{1}{p}+k\le \frac{x}{p}<\frac{1}{p}+k+1.
$$
Iterative application of \eqref{eq:Gamma:ind} leads to
\begin{equation}\label{eq:Gamma_3}
\Gamma(1+x/p)=\Gamma(x/p-k+1)\prod_{j=0}^{k-1}(x/p-j)\le C_p(x/p)^k,
\end{equation}
where $C_p=\sup_{1/p+1\le t\le 1/p+2}\Gamma(t)$. We conclude, that
$$
\Gamma(1+x/p)^{1/x}x^{-1/p}\le C_p^{1/x}p^{-k/x}x^{k/x-1/p}\le C_p^{1/x}\max(1,p^{-1/p})x^{-1/px},
$$
which is bounded (by a constant depending on $p$) for $x\in[1,\infty).$

On the other hand, using the first identity in \eqref{eq:Gamma_3} and Riemann sums, we obtain
\begin{align*}
\log\biggl(\frac{\Gamma(1+x/p)^{1/x}}{x^{1/p}}\biggr)&=\frac{1}{x}\log\Gamma(1+x/p)-\frac{1}{p}\log(x)\\
&\ge \frac{1}{x}\log(c_p)+\frac{1}{x}\sum_{j=0}^{k-1}\log(x/p-j)-\frac{1}{p}\log(x)\\
&\ge \frac{1}{x}\log(c_p)+\frac{1}{x}\int_{x/p-k}^{x/p}\log(t)dt-\frac{1}{p}\log(x)\\
&=\frac{1}{x}\log(c_p)+\frac{1}{p}\log(1/p)-\frac{1}{p}-\frac{1}{x}f(x/p-k),
\end{align*} 
where $c_p=\inf_{1/p+1\le t\le 1/p+2}\Gamma(t)>0$ and $f(t)=t\log(t)-t.$ 
The last expression is bounded (by a constant depending on $p$) for $x\in[1,\infty)$, as $1/p\le x/p-k<1/p+1.$
\end{proof}

To apply the volume arguments, it is of course necessary to calculate (or at least to estimate) the volume of the unit ball $B_p^n$ in $\ell_p^n(\R)$.
The exact value has been known (at least) since the work of Dirichlet, cf. \cite{Dirichlet}. We give a more modern proof, cf. \cite{Pisier}.

\begin{thm}\label{thm:vol_ball}
Let $0<p\le\infty$ and let $n\in\N$. Then it holds
\begin{equation}\label{eq:vol}
\vol(B^n_p)=\frac{2^n\cdot\Gamma(1+\frac{1}{p})^n}{\Gamma(1+\frac{n}{p})}.
\end{equation}
\end{thm}
\begin{proof}
Let $f$ be a smooth non-increasing positive function on $[0,\infty)$ quickly decaying to zero at infinity.
Then, by Fubini's theorem and partial integration,
\begin{align*}
\notag \int_{\R^n}f(\|x\|_p)dx&=-\int_{\R^n}\int_{\|x\|_p}^\infty f'(t)dtdx=
-\int_0^\infty \biggl(\int_{x:\|x\|_p\le t}1\,dx\biggr) f'(t)dt\\
&=-\int_0^\infty {\rm vol}(tB_p^n)f'(t)dt=-{\rm vol}(B_p^n)\int_0^\infty t^nf'(t)dt\\
&={\rm vol}(B_p^n)\cdot\int_0^\infty nt^{n-1}f(t)dt.
\end{align*}

For $f(t)=e^{-t^p}$, we get
\begin{align}
\notag\int_{\R^n}e^{-\|x\|_p^p}dx&={\rm vol}(B_p^n)\cdot\int_0^\infty nt^{n-1}e^{-t^p}dt={\rm vol}(B_p^n)\cdot \frac{n}{p}\cdot \int_0^\infty s^{n/p-1}e^{-s}ds\\
\label{eq:FB1}&=\frac{n\,{\rm vol}(B_p^n)\Gamma(n/p)}{p}={\rm vol}(B_p^n)\Gamma(1+n/p).
\end{align}
The proof is then finished by Fubini's theorem
\begin{align*}
{\rm vol}(B_p^n)\Gamma(1+n/p)&=\int_{\R^n}e^{-\|x\|_p^p}dx=\int_{\R^n}e^{-|x_1|^p-\dots-|x_n|^p}dx=
\biggl(2\int_0^\infty e^{-t^p}dt\biggr)^n\\
&=2^n\biggl(\frac{1}{p}\int_0^\infty s^{1/p-1}e^{-s}ds\biggr)^n=2^n\biggl(\frac{\Gamma(1/p)}{p}\biggr)^n=2^n\Gamma(1+1/p)^n.
\end{align*}
\end{proof}

Theorem \ref{thm:vol_ball} combined with Lemma \ref{lem:Gamma} gives
\begin{equation}\label{eq:Gamma:eq}
\vol(B_p^n)^{1/n}=\frac{2\Gamma(1+\frac{1}{p})}{\Gamma(1+\frac{n}{p})^{1/n}}\sim n^{-1/p},\quad n\ge 1,
\end{equation}
where the constants of equivalence depend again on $p$.

Next, we collect two simple facts about the $\ell_p^n$-(quasi-)norms. The easy proof is left to the reader.
\begin{lem}\label{lem:norm}
\begin{itemize}
\item[(i)] Let $0<p,q\le\infty$ and $n\in\N$. Then
\begin{equation*}
\|id:\ell_p^n(\R)\to\ell_q^n(\R)\|=\max(1,n^{1/q-1/p}).
\end{equation*}
\item[(ii)] Let $0<p<q<\infty$. Then $\|x\|_q\le \|x\|^{p/q}_p\cdot\|x\|^{1-p/q}_\infty$.
\end{itemize}
\end{lem}

The following lemma is the analogue of \cite[Proposition 12.1.13]{Pietsch} for quasi-Banach spaces and appeared already in \cite{HKV}.
Although it uses the volume arguments, no calculation of $\vol(B_X)$ is necessary, because two such terms cancel each other out.
The estimate obtained is optimal up to the constant $4^{1/p}$, cf. Section \ref{sec:4.1} for details.
\begin{lem}\label{lem:entropy_finite}
        Let $0<p\leq 1$, $n\in \N$, and let $X$ be a real $n$-dimensional $p$-Banach space. Then, for all $k\in\N$,
        \begin{equation}\label{eq:idX}
        e_k(id:X\to X) \le 4^{1/p}\cdot 2^{-\frac{k-1}{n}}.
        \end{equation}
\end{lem}
\begin{proof} If $k-1\le 2n/p$, the result is trivial as the right-hand side of \eqref{eq:idX} is larger than or equal to 1.
We will therefore assume that $k-1>2n/p$ or, equivalently, $2^{(k-1)p/n}>4.$

Let $x_1,\dots,x_N$ be a maximal family of elements of $B_X$ with $\|x_i-x_j\|_X>\tau$ for $i\not=j$, where $\tau>0$ is given by
$\frac{1+\tau^p/2}{\tau^p/2}=2^{p(k-1)/n}$.
Then, by triangle inequality, the sets $x_i+2^{-1/p}\tau B_X$ are mutually disjoint, they are all included in $(1+\tau^p/2)^{1/p}B_X$ and $B_X$ is covered by the union
of $x_i+\tau B_X$ over $i=1,\dots,N$. By volume comparison, we get
$$
N\vol(2^{-1/p}\tau B_X)\le \vol[(1+\tau^p/2)^{1/p}B_X]
$$
and, therefore,
$$
N\le \frac{(1+\tau^p/2)^{n/p}}{2^{-n/p}\tau^n}=\Bigl(\frac{1+\tau^p/2}{\tau^p/2}\Bigr)^{n/p}=2^{k-1}.
$$
We conclude that
$$
e_k(id:X\to X)\le \tau=\biggl[\frac{2}{2^{(k-1)p/n}-1}\biggr]^{1/p}\le\biggl[\frac{4}{2^{(k-1)p/n}}\biggr]^{1/p}=4^{1/p}2^{-\frac{k-1}{n}}.\qedhere
$$
\end{proof}

The behavior of entropy numbers with respect to interpolation of Banach spaces was studied intensively.
It is rather easy to show, that they behave well if one of the endpoints is fixed and we refer to \cite{ET} for details.
Surprisingly, it was shown only recently in \cite{EN}, that a general interpolation formula for entropy numbers with both endpoints
interpolated is out of reach. We give only a simplified version in a form, which shall be needed later on.

\begin{lem}\label{lem:interpol}
Let $0<p\le q< \infty$ and let $k,l,n$ be positive integers. Then
$$
e_{k+l-1}(id:\ell_p^n(\R)\to\ell_q^n(\R))\le 2^{1/\bar p}e^{p/q}_k(id:\ell_p^n(\R)\to\ell_p^n(\R))\cdot e_l^{1-p/q}(id:\ell_p^n(\R)\to\ell_\infty^n(\R)),
$$
where $\bar p=\min(1,p).$
\end{lem}
\begin{proof} Let $\varepsilon>0$ be arbitrary. We put
$$
e^0:=(1+\varepsilon)e_k(id:\ell_p^n(\R)\to\ell_p^n(\R))\quad\text{and}\quad e^1:=(1+\varepsilon)e_l(id:\ell_p^n(\R)\to\ell_\infty^n(\R)).
$$ 
Then $B_p^n$ can be covered by $2^{k-1}$ balls in $\ell_p^n(\R)$ with radius $e^0$ and by $2^{l-1}$ balls in $\ell_\infty^n(\R)$ with radius $e^1.$
We can therefore decompose
$$
B_p^n=\bigcup_{i}A_i,
$$
where this union 
contains at most $2^{k-1}$ components $A_i$, and each $A_i$ lies in some ball in $\ell_p^n(\R)$ with radius $e^0.$
Similarly we can write
$$
B_p^n=\bigcup_{j}B_j,
$$
with at most $2^{l-1}$ components $B_j$, each of them lying in some ball in $\ell_\infty^n(\R)$ with radius $e^1.$

Finally, we denote $C_{i,j}=A_i\cap B_j$ and choose $z_{i,j}\in C_{i,j}$ arbitrarily any time $C_{i,j}$ is non-empty.
Let now $x\in C_{i,j}.$ Then both $x$ and $z_{i,j}$ are in $C_{i,j}\subset A_{i}$. Therefore $\|x-z_{i,j}\|_p\le 2^{1/\bar p}e^0$.
Using $C_{i,j}\subset B_j$, we get in the same way also $\|x-z_{i,j}\|_\infty\le 2e^1\le 2^{1/\bar p}e^1$. Hence, by Lemma \ref{lem:norm},
$$
\|x-z_{i,j}\|_q\le 2^{1/\bar p}(e^0)^{p/q}(e^1)^{1-p/q},
$$
and balls with centers in $z_{i,j}$'s with this radius in $\ell_q^n(\R)$ cover $B_p^n$. Finally, there is at most $2^{k+l-2}$ such points.
\end{proof}

Some of the arguments used in the proof of Theorem \ref{hlavna1''} are of combinatorial nature.
As a preparation, we present the following lemma from \cite{K2001}, another variant is discussed in the last section.

\begin{lem}\label{lem:comb}
Let $m,n\in\N$ with $m\le \frac{n}{4}$ and define
$$
H_m=\{x\in\{-1,0,1\}^n:\|x\|_1=2m\}.
$$
Furthermore, for $x,y\in H_m$, define their \emph{Hamming distance} as $h(x,y)=\#\{i:x_i\not=y_i\}$.
Then there is a set $A_m\subset H_m$ with $\#A_m\ge [n/(2m)]^m$, such that any two distinct $x,y\in A_m$
satisfy $h(x,y)> m.$
\end{lem}
\begin{proof}First note, that
\begin{equation}\label{eq:comb1}
\#H_m={n\choose 2m}2^{2m}.
\end{equation}
Second, if $x\in H_m$ is arbitrary, then the set $\{y\in H_m:h(x,y)\le m\}$ has cardinality at most
\begin{equation}\label{eq:comb2}
{n\choose m}3^{m}.
\end{equation}
Indeed, there is ${n\choose m}$ ways to choose $m$ coordinates, where $x$ and $y$ may differ, and three possible values for each of the
coordinates for $y$.

The set $A_m$ can be constructed by the following greedy algorithm. First, choose $x^1\in H_m$ arbitrarily. If $x^1,\dots,x^l$ were already selected
and if there is some $y\in H_m$, which has the Hamming distance from all these points at least equal to $m+1$, put $x^{l+1}:=y$. Otherwise, the algorithm stops
and $A_m$ is the collection of all $x^1,\dots,x^N$, which were selected so far. This ensures, that $h(x,y)\ge m+1$ for any two distinct $x,y\in A_m$.

Finally, from \eqref{eq:comb1} and \eqref{eq:comb2} it follows that
$$
\#A_m=N\ge \frac{{n\choose 2m}2^{2m}}{{n\choose m}3^{m}}=\frac{4^m}{3^m}\cdot \frac{(n-m)\dots(n-2m+1)}{(2m)\dots(m+1)}
\ge \frac{4^m}{3^m}\Bigl(\frac{n-m}{2m}\Bigr)^m\ge \Bigl(\frac{n}{2m}\Bigr)^m.
$$
We have used that $t\to \frac{n-2m+t}{m+t}$ is decreasing on $(0,\infty)$ and $n\ge 4m$.
\end{proof}

\section{Proof of Theorem \ref{hlavna1''}}\label{sec:3}
This section is devoted to the proof of the main result, Theorem \ref{hlavna1''}.
For reader's convenience, we repeat its statement and then prove step-by-step
all the upper and lower bounds of $e_k (id:\ell^n_p(\mathbb{R})\rightarrow \ell^n_q(\mathbb{R}))$
for all possible values of $p,q,k$ and $n$.
\setcounter{thm}{1}
\setcounter{equation}{1}
\begin{thm}
Let $0<p,q\leq \infty$ and let $n\in\N$.
\begin{itemize}
\item[a)] If $0<p\leq q\leq \infty $ then for all $k\in \mathbb{N}$ it holds
\begin{equation}
e_k (id:\ell^n_p(\mathbb{R})\rightarrow \ell^n_q(\mathbb{R}))\sim 
\begin{cases} 1 & \text{if} \quad 1\leq k\leq \log_2n,\\
\displaystyle\biggl(\frac{\log_2(1+n/k)}{k}\biggr)^{\frac{1}{p}-\frac{1}{q}} & \text{if} \quad \log_2n\leq k\leq n,\\[10pt]
\displaystyle2^{-\frac{k-1}{n}}n^{\frac{1}{q}-\frac{1}{p}} & \text{if} \quad n\leq k.
\end{cases}
\end{equation}
\item[b)] If $0<q\leq p\leq \infty$ then for all $k\in \mathbb{N}$ it holds
\begin{equation}
 e_k (id:\ell^n_p(\mathbb{R})\rightarrow \ell^n_q(\mathbb{R}))\sim  2^{-\frac{k-1}{n}}n^{\frac{1}{q}-\frac{1}{p}}.
\end{equation}
\end{itemize}
The constants of equivalence in both \eqref{eq:2ekviv1''} and \eqref{eq:2ekviv2''} may depend on $p$ and $q$, but are independent of $k$ and $n$.
\end{thm}
\setcounter{thm}{9}
\setcounter{equation}{13}
\begin{proof} We split the proof of the different estimates in Theorem \ref{hlavna1''} by the used technique.
Furthermore, we denote throughout the proof $\bar p=\min(1,p)$ and $\bar q=\min(1,q).$\vskip.5cm

\noindent\underline{\emph{Step 1: Elementary estimates}}\\
\begin{enumerate}
\item[(i)] If $0<p\le q\le \infty$, we have by Theorem \ref{thm:entropy} and Lemma \ref{lem:norm}
$$
e_k (id:\ell^n_p(\R)\rightarrow \ell^n_q(\R))\le \|id:\ell^n_p(\R)\rightarrow \ell^n_q(\R)\|
=1
$$
or, equivalently, we can always cover $B_p^n$ with $B_q^n.$ This gives the optimal upper bound for $1\le k\le \log_2n.$

\item[(ii)] On the other side, the canonical vectors $e^1,\dots,e^n$ defined by
$$
(e^i)_l=\delta_{il},\quad i,l=1,\dots,n,
$$
satisfy $e^i\in B_p^n$ and $\|e^i-e^j\|_q=2^{1/q}$ for $i,j\in\{1,\dots,n\}$ and $i\not=j.$ Therefore, if we cover $B_p^n$
with $2^{k-1}<n$ balls in $\ell_q^n(\R)$ of radius $r>0$, at least one of these balls must contain two different $e^i,e^j$ with $i\not=j$,
simply by the pigeonhole principle. We denote the center of such a ball by $z\in\R^n$ and obtain by triangle inequality
$$
2^{\bar q/q}= \|e^i-e^j\|^{\bar q}_q\le \|e^i-z\|^{\bar q}_q+\|z-e^j\|^{\bar q}_q\le 2r^{\bar q}.
$$
Therefore,
$$
e_k (id:\ell^n_p(\R)\rightarrow \ell^n_q(\R))\ge 2^{1/q-1/\bar q}
$$
for $1\le k\le \log_2n.$ Together with (i), this finishes the proof of the equivalence in \eqref{eq:2ekviv1''} for this range of $k$'s.
\item[(iii)] Lemma \ref{lem:entropy_finite} together with Lemma \ref{lem:norm} can be used to prove the upper bound in \eqref{eq:2ekviv2''}.
Let $0<q\le p\le \infty$ and let $k$ and $n$ be positive integers. Then
\begin{align*}
e_k(id:\ell^n_p(\R)\rightarrow \ell^n_q(\R))&\le e_k(id:\ell^n_p(\R)\rightarrow \ell^n_p(\R))\cdot
\|id:\ell^n_p(\R)\rightarrow \ell^n_q(\R)\|\\
&\le 4^{1/p}\cdot 2^{-\frac{k-1}{n}}\cdot n^{1/q-1/p}.
\end{align*}
\end{enumerate}

\noindent\underline{\emph{Step 2: Volume arguments, estimates from below}}

If $B_p^n$ is covered by $2^{k-1}$ balls in $\ell_q^n(\R)$ of radius $r>0$, then their volume must be larger than the volume of $B_p^n$.
This simple observation can be turned into the estimate
\begin{equation*}
\vol(B_p^n)\le 2^{k-1}\vol(r B_q^n)=2^{k-1}r^n\vol(B_q^n)
\end{equation*}
and, by \eqref{eq:Gamma:eq},
\begin{equation*}
e_k(id:\ell^n_p(\R)\rightarrow \ell^n_q(\R))\ge 2^{-\frac{k-1}{n}}\biggl(\frac{\vol(B_p^n)}{\vol(B_q^n)}\biggr)^{1/n}\sim 2^{-\frac{k-1}{n}}n^{1/q-1/p}.
\end{equation*}
This proves the lower bound in \eqref{eq:2ekviv2''} for all $k$'s and in \eqref{eq:2ekviv1''} for $k\ge n.$
Together with Step 1, the proof of \eqref{eq:2ekviv2''} is therefore finished.

\vskip.5cm\noindent\underline{\emph{Step 3: Volume arguments, estimates from above}}

Let $0<p\le q\le\infty$ and let $\tau>0$. Similarly to the proof of Lemma \ref{lem:entropy_finite},
we let $\{y_1,\dots,y_N\}\subset B_p^n$ be any maximal $\tau$-distant set in the $\ell_q^n$-(quasi)-norm. In detail, this means that we assume that
$\|y_i-y_j\|_q> \tau$ for $i,j\in\{1,\dots,N\}$ with $i\not=j$
and that for every $y\in B_p^n$, there is $i\in\{1,\dots,N\}$ with $\|y-y_i\|_q\le\tau$.

Then we observe a couple of simple facts 
\begin{enumerate}
\item[(i)] If $z\in (y_i+2^{-1/\bar q}\tau B_q^n)\cap (y_j+2^{-1/\bar q}\tau B_q^n)$ for $i\not=j$, then
$$
\tau^{\bar q}< \|y_i-y_j\|^{\bar q}_q\le \|y_i-z\|_q^{\bar q}+\|z-y_j\|_q^{\bar q}\le 2 (2^{-1/\bar q}\tau)^{\bar q},
$$
gives a contradiction. Hence, the sets $y_j+2^{-1/\bar q}\tau B_q^n, j=1,\dots,N,$ are disjoint.
\item[(ii)] If $z\in y_j+2^{-1/\bar q}\tau B_q^n$, we obtain for $\nu=\|id:\ell_q^n(\R)\to\ell_p^n(\R)\|=n^{1/p-1/q}$
$$
\|z\|^{\bar p}_p\le\|y_j\|_p^{\bar p}+\|z-y_j\|_p^{\bar p}\le 1+ \nu^{\bar p}\|z-y_j\|_q^{\bar p}\le 1+(2^{-{1/\bar q}}\tau\nu)^{\bar p}.
$$
It follows that all the sets $y_j+2^{-1/\bar q}\tau B_q^n$, $j=1,\dots,N$ are included in $(1+(2^{-1/\bar q}\tau\nu)^{\bar p})^{1/\bar p} B_p^n$.
\end{enumerate}
By volume comparison, we conclude
\begin{align*}
N\vol(2^{-1/\bar q}\tau B_q^n)&\le \vol[(1+(2^{-{1/\bar q}}\tau\nu)^{\bar p})^{1/\bar p}B_p^n]
\end{align*}
and
\begin{equation}\label{eq:proof1}
N\le \frac{(1+(2^{-{1/\bar q}}\tau\nu)^{\bar p})^{n/\bar p}\vol(B_p^n)}{2^{-n/\bar q}\tau^{n} \vol(B_q^n)}
=\biggl[\frac{1+(2^{-{1/\bar q}}\tau\nu)^{\bar p}}{2^{-\bar p/\bar q}\tau^{\bar p} }\biggr]^{n/\bar p}V_n(p,q),
\end{equation}
where we denoted $V_n(p,q)=\frac{\vol(B_p^n)}{\vol(B_q^n)}$. For a positive integer $k$, we define $\tau$
by putting the right-hand side of \eqref{eq:proof1} equal to $2^{k-1}$. We obtain
$$
\biggl[\frac{1+(2^{-{1/\bar q}}\tau\nu)^{\bar p}}{2^{-\bar p/\bar q}\tau^{\bar p} }\biggr]^{n/\bar p}V_n(p,q)=2^{k-1}\quad
\text{and}\quad
\tau=\frac{2^{1/\bar q}}{[2^{(k-1)\bar p/n}V_n(p,q)^{-\bar p/n}-\nu^{\bar p}]^{1/\bar p}}.
$$
Observe, that this is always possible if the denominator is positive. This is the case if
\begin{equation}\label{eq:temp1}
2^{\frac{k-1}{n}}>2\nu V_n(p,q)^{1/n}=2n^{1/p-1/q}\biggl(\frac{\vol(B_p^n)}{\vol(B_q^n)}\biggr)^{1/n}.
\end{equation}
By \eqref{eq:Gamma:eq}, the right-hand side of \eqref{eq:temp1} is equivalent to a constant, and \eqref{eq:temp1} is satisfied if $k\ge \gamma_{p,q}n$
for some $\gamma_{p,q}>0$ depending only on $p$ and $q$.

We conclude that, for $k\ge \gamma_{p,q} n$,
\begin{align*}
e_k(id:\ell_p^n(\R)\to\ell_q^n(\R))\le \tau
&\le \frac{C_{p,q}}{[2^{(k-1)\bar p/n}V_n(p,q)^{-\bar p/n}]^{1/\bar p}}\le C'_{p,q}2^{-\frac{k-1}{n}}n^{1/q-1/p},
\end{align*}
where we used \eqref{eq:temp1} in the last but one inequality with $C_{p,q}=2^{1+1/\bar q}/(2^{\bar p}-1)^{1/\bar p}.$

\vskip.5cm\noindent\underline{\emph{Step 4: Combinatorial part - estimate from below}}

Let $0<p\le q\le\infty.$ Let $m,n\in\N$ with $m\le n/4$ and let $A_m$ be the set from Lemma \ref{lem:comb}.
Then $\tilde A_m:=(2m)^{-1/p}A_m\subset B_p^n$ and for two distinct $x,y\in \tilde A_m$ it holds $\|x-y\|_q\ge 2^{-1/p}m^{1/q-1/p}.$
As a consequence, two different points from $\tilde A_m$ cannot be covered by an $\ell_q^n$-ball, unless its radius is at least
$2^{-1/p-1/\bar q}m^{1/q-1/p}.$

We summarize, that if for some $k\in\N$, there is $m\in\N$ with 
\begin{equation}\label{eq:comb_1}
m\le n/4\quad \text{and}\quad 2^{k-1}<\Bigl(\frac{n}{2m}\Bigr)^m,
\end{equation}
then $e_k(id:\ell_p^n\to\ell_q^n)\ge 2^{-1/p-1/\bar q}m^{1/q-1/p}.$

Let us therefore fix $k,n\in\N$ with $n\ge 64$ and $\log_2 n\le k\le n/8.$ Then
$k\ge\log_2n\ge\log_2(n/k+1)$
and we may choose $m$ to be any integer with
$$
\frac{k}{\log_2(n/k+1)}\le m\le \frac{2k}{\log_2(n/k+1)}.
$$
Then $m\le 2k/\log_2(8)=2k/3\le n/12$. The function $f:t\to t\log_2(n/(2t))$ is increasing on $(0,\frac{n}{2e})$ and therefore
\begin{align*}
m\log_2\Bigl(\frac{n}{2m}\Bigr)&=f(m)\ge f\Bigl(\frac{k}{\log_2(n/k+1)}\Bigr)\\
&=\frac{k}{\log_2(n/k+1)}\log_2\Bigl(\frac{n\log_2(n/k+1)}{2k}\Bigr)\ge k,
\end{align*}
where in the last inequality we used that
$$
\frac{n}{2k}\log_2(n/k+1)\ge \frac{3n}{2k}\ge \frac{n}{k}+1.
$$
We therefore get also
$$
2^{k}\le \Bigl(\frac{n}{2m}\Bigr)^m.
$$
This finishes the proof of \eqref{eq:comb_1}. It follows that
$$
e_k(id:\ell_p^n(\R)\to\ell_q^n(\R))\ge 2^{-1/p-1/\bar q}m^{1/q-1/p}\ge c_{p,q}\Bigl(\frac{\log_2(n/k+1)}{k}\Bigr)^{1/p-1/q}
$$
for $n\ge 64$ and $\log_2 n\le k\le n/8.$ Furthermore, by Step 2, we know that $e_n(id:\ell_p^n(\R)\to\ell_q^n(\R))\gtrsim n^{1/q-1/p}.$
This allows to use the monotonicity of entropy numbers and to obtain the lower bound in \eqref{eq:2ekviv1''} for all $n\ge 64$ and all positive integers $k$.
The (finitely many) remaining values of $n$ can then also be incorporated at the cost of possibly larger constants.

\vskip.5cm\noindent\underline{\emph{Step 5: Combinatorial part - estimate from above}}

We first prove the result for $0<p<q=\infty$, the general case then follows by interpolation, i.e. by Lemma \ref{lem:interpol}.
Let $1\le m\le n$ be two natural numbers. Then every $x\in B_p^n$ can be approximated in the $\ell_\infty^n$-norm
by a vector $\tilde x$ with at most $m$ non-zero components and $\|x-\tilde x\|_\infty\le r^0:=m^{-1/p}.$ Indeed, we can take $\tilde x$
equal to $x$ on the indices of the $m$ largest (in the absolute value) components of $x$, and zero elsewhere.

Let now $l\ge \gamma_{p,\infty}m$ be an integer, where $\gamma_{p,\infty}>0$ is the constant defined at the end of Step 3.
Then we know that $e_l(id:\ell_p^m(\R)\to\ell_\infty^m(\R))\le r^1:=C_{p,\infty}2^{-\frac{l-1}{m}}m^{-1/p}.$
Therefore, for any $\varepsilon>0$, there exist points $x_1,\dots,x_{2^{l-1}}\in\R^m$ with
$$
B_p^m\subset \bigcup_{j=1}^{2^{l-1}}(x_j+(1+\varepsilon)r^1 B_\infty^m).
$$
we collect the points with support of the size at most $m$ and equal to some of $x_j$'s on its support. In this way, we obtain
at most ${n\choose m}2^{l-1}$ centers of $\ell_\infty^n$-balls of radius $r^0+(1+\varepsilon)r^1$, which cover $B_p^n.$

We conclude, that if $1\le m\le n$ and $l\ge \gamma_{p,\infty}m$, then
\begin{equation}\label{eq:temp2}
2^{k-1}\ge {n\choose m}2^{l-1}\quad \implies\quad e_k(id:\ell_p^n(\R)\to\ell_\infty^n(\R))\le m^{-1/p}\Bigl(1+C_{p,\infty}2^{-\frac{l-1}{m}}\Bigr).
\end{equation}
For $1\le m\le n$, we choose any integer $l$ with $\gamma_{p,\infty}m\le l\le (\gamma_{p,\infty}+1)m$.
Using the elementary estimate
\begin{align*}
{n\choose m}\le \frac{n^m}{m!}\le n^m \Bigl(\frac{e}{m}\Bigr)^m\le\Bigl(\frac{3n}{m}\Bigr)^m\le \Bigl(\frac{n}{m}+1\Bigr)^{3m}
\end{align*}
we obtain
$$
\log_2\biggl[{n\choose m}2^{l-1}\biggr]= l-1+\log_2{n\choose m}\le (\gamma_{p,\infty}+1)m-1+3m\log_2\Bigl(\frac{n}{m}+1\Bigr).
$$
Together with \eqref{eq:temp2}, we arrive at
$$
k\ge \gamma_pm\log_2(n/m+1)\quad \implies\quad
e_k(id:\ell_p^n(\R)\to\ell_\infty^n(\R))\le C_pm^{-1/p},
$$
where $\gamma_p=\gamma_{p,\infty}+4$ and $C_p=1+C_{p,\infty}$ depend only on $p$.

Next, we define $\alpha_p>2$ to be a real number large enough to ensure
\begin{equation}\label{eq:last1}
\gamma_p\alpha_p+1\le 2^{\alpha_p-2}.
\end{equation}
Moreover, we assume that $k$ and $n$ are positive integers with $n\ge k\ge 2\alpha_p\gamma_p\log_2n$. 
This allows us to choose $m$ to be any integer with
\begin{align}\label{eq:comb:11}
\frac{k}{ 2\alpha_p\gamma_p\log_2(n/k+1)}\le m\le \frac{k}{\alpha_p\gamma_p\log_2(n/k+1)}.
\end{align}
From $\gamma_p>4$ and $\alpha_p>2$, it follows that $m\le n$ and, by monotonicity of $f:t\to t\log_2(n/t+1)$ on $(0,\infty)$, we get
\begin{align*}
\gamma_p f(m)&=\gamma_pm\log_2(n/m+1)\le \gamma_p f\biggl(\frac{k}{\gamma_p \alpha_p\log_2(n/k+1)}\biggr)\\
&=\gamma_p \cdot \frac{k}{\gamma_p \alpha_p\log_2(n/k+1)}\log_2\Bigl(\frac{n}{k}\gamma_p \alpha_p\log_2\Bigl(\frac{n}{k}+1\Bigr)+1\Bigr)\\
&\le\frac{k}{\alpha_p\log_2(n/k+1)}\Bigl[\log_2\Bigl(\frac{n}{k}+1\Bigr)+\log_2(\gamma_p \alpha_p+1)+\log_2\Bigl(\log_2\Bigl(\frac{n}{k}+1\Bigr)+1\Bigr)\Bigr]\\
&\le \frac{2k}{\alpha_p}+\frac{k\log_2(\gamma_p\alpha_p+1)}{\alpha_p\log_2(n/k+1)}\le \frac{k}{\alpha_p}\Bigl[2+\log_2(\gamma_p\alpha_p+1)\Bigr]\le k, 
\end{align*}
where we used \eqref{eq:last1} in the last step.

This finishes the proof of the upper bound in \eqref{eq:2ekviv1''} for $q=\infty$ and $2\alpha_p\gamma_p\log_2n\le k\le n.$
The $k$'s between $\log_2n$ and $2\alpha_p\gamma_p\log_2n$ can be incorporated by monotonicity at the cost of larger constants.

Finally, the case $0<p<q<\infty$ follows by interpolation. Indeed, Lemma \ref{lem:interpol} implies that
\begin{equation*}
e_k(id:\ell_p^n(\R)\to\ell_q^n(\R))\le 2^{1/\bar p}e^{1-\frac{p}{q}}_k(id:\ell_p^n(\R)\to\ell_\infty^n(\R)),
\end{equation*}
and the result follows from the bound just proven for $e_k(id:\ell_p^n(\R)\to\ell_\infty^n(\R))$.
\end{proof}

\section{Extensions}\label{sec:4}

We collect further facts about the entropy numbers of identities between finite-dimensional spaces, which seem to be very well known in the community.
Several parts of Theorem \ref{hlavna1''} can be proved in different ways and we present some of these alternative approaches. Finally,
we extend Theorem \ref{hlavna1''} also to the complex setting.

\subsection{Alternative use of volume arguments}\label{sec:4.1}

Without much work, one can show that Lemma \ref{lem:entropy_finite} is optimal up to the constant $4^{1/\bar p}$. Indeed,
if $X$ is a $n$-dimensional real vector spaces equipped with a (quasi)-norm and $B_X$ is covered by $2^{k-1}$ translations of $rB_X$,
then
$$
\vol(B_X)\le 2^{k-1}r^n\vol(B_X).
$$
This implies that $e_k(id:X\to X)\ge 2^{-\frac{k-1}{n}}.$

If $0<p\le q\le \infty$, we can write
$$
2^{-\frac{k-1}{n}}\le e_k(id:\ell_p^n(\R)\to\ell_p^n(\R))\le e_k(id:\ell_p^n(\R)\to\ell_q^n(\R))\cdot\|id:\ell_q^n(\R)\to\ell_p^n(\R)\|.
$$
By Lemma \ref{lem:norm}, the last norm is equal to $n^{1/p-1/q}$ and we obtain the lower bound in \eqref{eq:2ekviv1''} for $k\ge n$.


The upper bound on $e_k(id:\ell^n_p(\R)\rightarrow \ell^n_q(\R))$ for $0<p<q\le\infty$ and $\log_2 n\le k\le n$
was proven in Section \ref{sec:3} first for $q=\infty$ and, afterwards, for $p<q<\infty$ by interpolation. Actually, one can adapt
the original proof for $q=\infty$ also to the general setting of $q\le\infty.$ Unfortunately, this comes at a price of further technical
difficulties and we, together with \cite{Pietsch}, preferred to go through interpolation.

\subsection{Alternative use of combinatorial arguments}

Next comment considers the lower bound in \eqref{eq:2ekviv1''} for $\log_2 n\le k\le n$. Lemma \ref{lem:comb}
constructs a large subset (namely $(2m)^{-1/p}A_m$) of $B_p^n$ with the elements having large mutual distances in $\ell_q^n.$
There is, however, more ways to achieve a similar result. We present a statement, which is very well-known in coding theory, cf. \cite{coding:1,coding:2},
and appeared also in \cite{FPRU,FR} and \cite{BCKV} in connection with Gelfand widths and optimality results of sparse recovery.
There the reader can also find a short proof, which is very much in the spirit of the proof of Lemma \ref{lem:comb}.

\begin{lem}\label{lem:hol} Let $k\le n$ be two positive integers. Then there are $N$ subsets $T_1,\dots,T_N$ of $\{1,\dots,n\}$, such that
\begin{enumerate}
\item[(i)] $\displaystyle N\ge \Bigl(\frac{n}{4k}\Bigr)^{k/2}$,\vskip.2cm
\item[(ii)] $|T_i|=k$ for all $i=1,\dots,N$ and \vskip.2cm
\item[(iii)] $|T_i\cap T_j|<k/2$ for all $i\not =j.$
\end{enumerate}
\end{lem}
Similarly to Lemma \ref{lem:comb}, Lemma \ref{lem:hol} can be used to produce a large set of points in $B_p^n$ with large distance in $\ell_q^n.$
It is enough to take $x_j=k^{-1/p}\chi_{T_j}$, $j=1,\dots,N$, where $T_j$'s are the sets from Lemma \ref{lem:hol} and $\chi_{T_j}$
is the indicator vector of $T_j$.

\subsection{Complex spaces}
It is surprisingly easy to extend the estimates of Theorem \ref{hlavna1''} to the setting of complex spaces $\ell_p^n(\C).$
The main result then reads as follows.
\begin{thm}\label{hlavna2}
Let $0<p,q\leq \infty$ and let $n\in\N$.
\begin{itemize}
\item[a)] If $0<p\leq q\leq \infty $ then for all $k\in \mathbb{N}$ it holds
\begin{equation}\label{eq:2ekviv1'}
e_k (id:\ell^n_p(\mathbb{C})\rightarrow \ell^n_q(\mathbb{C}))\sim 
\begin{cases} 1 & \text{if} \quad 1\leq k\leq \log_2(2n),\\
\displaystyle \biggl(\frac{\log_2(1+2n/k)}{k}\biggr)^{\frac{1}{p}-\frac{1}{q}} & \text{if} \quad \log_2(2n)\leq k\leq 2n,\\[10pt]
\displaystyle 2^{-\frac{k-1}{2n}}n^{\frac{1}{q}-\frac{1}{p}} & \text{if} \quad 2n\leq k.
\end{cases}
\end{equation}
\item[b)] If $0<q\leq p\leq \infty$ then for all $k\in \mathbb{N}$ it holds
\begin{equation}\label{eq:2ekviv2'}
 e_k (id:\ell^n_p(\mathbb{C})\rightarrow \ell^n_q(\mathbb{C}))\sim  2^{-\frac{k-1}{2n}}n^{\frac{1}{q}-\frac{1}{p}}.
\end{equation}
\end{itemize}
The constants of equivalence in both \eqref{eq:2ekviv1'} and \eqref{eq:2ekviv2'} may depend on $p$ and $q$, but are independent of $k$ and $n$.
\end{thm}
\begin{proof}
We observe, that the mapping
$$
{\mathcal J}(z)={\mathcal J}(z_1,\dots,z_n)=(\re(z_1),\im(z_1),\dots,\re(z_n),\im(z_n))
$$
is bounded from $\ell_p^n(\C)$ to $\ell_p^{2n}(\R)$ with the norm bounded by a quantity independent of $n$. The same is true about 
$$
{\mathcal J'}(x)={\mathcal J'}(x_1,\dots,x_{2n})=(x_1+ix_2,\dots,x_{2n-1}+ix_{2n})
$$
as a mapping from $\ell_p^{2n}(\R)$ to $\ell_p^{n}(\C)$.

Using the submultiplicativity of entropy numbers, we get
\begin{align*}
e_k (id:\ell^n_p(\mathbb{C})\rightarrow \ell^n_q(\mathbb{C}))&\le
\|{\mathcal J}:\ell^n_p(\mathbb{C})\rightarrow \ell^{2n}_p(\R)\|\\
&\qquad \cdot e_k(id:\ell^{2n}_p(\R)\rightarrow \ell^{2n}_q(\R))
\cdot \|{\mathcal J'}:\ell^{2n}_q(\R)\rightarrow \ell^{n}_q(\mathbb{C})\|\\
\intertext{and}
e_k (id:\ell^{2n}_p(\R)\rightarrow \ell^{2n}_q(\R))&\le
\|{\mathcal J'}:\ell^{2n}_p(\R)\rightarrow \ell^{n}_p(\mathbb{C})\|\\
&\qquad \cdot e_k(id:\ell^{n}_p(\mathbb{C})\rightarrow \ell^{n}_q(\mathbb{C}))
\cdot \|{\mathcal J}:\ell^{n}_q(\mathbb{C})\rightarrow \ell^{2n}_q(\R)\|.
\end{align*}
This can be summarized into
$$
e_k (id:\ell^n_p(\mathbb{C})\rightarrow \ell^n_q(\mathbb{C}))
\sim e_k (id:\ell^{2n}_p(\R)\rightarrow \ell^{2n}_q(\R))
$$
with constants of equivalence independent on $k$ and $n$. The result then follows from Theorem \ref{hlavna1''}.
\end{proof}
\vskip.5cm
\noindent{\bf Acknowledgement}\\[10pt]
The second author was supported by the grant P201/18/00580S of the Grant Agency of the Czech Republic and by the Neuron Fund for Support of Science.

\end{document}